\documentclass[11pt]{amsart}

\usepackage{amssymb,,mathtools}
\usepackage[english]{babel}
\usepackage{enumerate}
\usepackage{tikz}
\usepackage[all]{xy}
\usepackage{geometry}
\usepackage{multicol}
\usepackage{hyperref}
\hypersetup{
    colorlinks,%
    citecolor=blue,%
    filecolor=blue,%
    linkcolor=blue,%
    urlcolor=blue
}

\usepackage[utf8]{inputenc}
\usepackage{tikz-cd}
\usepackage[utf8]{inputenc}
\usepackage[T1]{fontenc}
\usepackage{amssymb}
\usepackage[all]{xy}
\usepackage{graphicx}
\usepackage{amsmath}
\usepackage{colonequals}
\usepackage{multicol}
\usepackage{amsthm}
\usepackage{hyperref}
\newtheorem{thm}{Theorem}
\newtheorem{lem}[thm]{Lemma}
\newtheorem{defi}[thm]{Definition}
\newtheorem{prop}[thm]{Proposition}
\newtheorem{cor}[thm]{Corollary}
\newtheorem{rmk}[thm]{Remark}
\newtheorem*{thm*}{Theorem}
\newtheorem{Q}{Question}

\theoremstyle{definition}
\newtheorem{ex}{Example}

\newcommand{\ZZ}{\mathbb{Z}}

\newcommand{\PP}{\mathbb{P}}

\definecolor{myorchid}{HTML}{a4538a}

\title{Automorphisms of quartic surfaces and Cremona transformations}
\author[D. Paiva]{Daniela Paiva}
\address{Instituto de Matem\'atica Pura e Aplicada - IMPA\\
        Estrada Dona Castorina, 110 \\
         Rio de Janeiro, RJ 22460-320, Brazil}
\email{da.paiva@impa.br}
\author[A. Quedo]{Ana Quedo}
\address{Instituto de Matem\'atica Pura e Aplicada - IMPA\\
        Estrada Dona Castorina, 110 \\
       Rio de Janeiro, RJ 22460-320, Brazil}
\email{ana.quedo@impa.br}

\begin{document}

\begin{abstract}
In this paper, we consider the problem of determining which automorphisms of a smooth quartic surface $S \subset \mathbb{P}^3$ are induced by a Cremona transformation of $\mathbb{P}^3$. We provide the first steps towards a complete solution of this problem when $\rho(S)=2$. In particular, we give several examples of quartics whose automorphism groups are generated by involutions, but no non-trivial automorphism is induced by a Cremona transformation of $\mathbb{P}^3$, giving a negative answer for Oguiso's question of whether every automorphism of finite order of a smooth quartic surface $S\subset \mathbb{P}^3$ is induced by a Cremona transformation. 
 
\end{abstract}
\maketitle
\section{Introduction}
Given a smooth quartic surface $S\subset \mathbb{P}^3$, Gizatullin posed the following problem \cite{Oguiso1}:

\begin{Q}
    Which automorphisms of $S$ are induced by Cremona transformations of $\mathbb{P}^3$?
\end{Q}

Matsumura and Monsky proved in \cite{Matsumura} that for $n\geq 2$, every automorphism of a smooth hypersurface of degree $d$ in $\mathbb{P}^{n+1}$ is induced by an automorphism of the ambient space, provided that $(n,d)\neq (2,4)$. So this question makes sense for smooth quartic surfaces in $\mathbb{P}^3$.\\

When the smooth quartic surface $S\subset\mathbb{P}^3$ has Picard rank $\rho(S)=1$ the problem is trivially solved since $Aut(S)=\{1\}$ (see for instance \cite[$\S15$, Corollary 2.12]{Huybrechts}). However, the general behaviour for higher Picard number is still unknown, despite known examples constructed by Oguiso. In fact, Oguiso was the first one to address this problem in \cite{Oguiso1} and \cite{Oguiso2}. In \cite{Oguiso1}, he constructed a smooth quartic K3 surface $S$ with $\rho(S)=2$ where no nontrivial element of $Aut(S)$ is induced by a Cremona transformation of $\mathbb{P}^3$. In this case $Aut(S)=\mathbb{Z}$, so every automorphism of $S$ has infinite order.
In particular, based on a more general result of Takahashi \cite{Takagashi}, Oguiso observed that no automorphism of a smooth quartic K3 surface is induced by an element of $Bir(\mathbb{P}^{3}) \backslash Aut(\mathbb{P}^{3})$ if the quartic surface does not contain curves of degree less than 16 that are not complete intersections of the surface with a hypersurface of $\mathbb{P}^{3}$ (Theorem \ref{oguiso}).\\

In this paper, we follow  Oguiso's strategy to study the case of smooth quartic surfaces $S\in\PP^3$ with $\rho(S)=2$. For a smooth quartic surface $S$, we denotee by $disc(S)$ the discriminant of $S$, i.e., the determinant of a matrix representing its
intersection product, and by $G\subset Aut(S)$ the finite index subgroup of all automorphisms of $S$ preserving the unique global $2$-form up to sign (see $\S$\ref{group G}). These objects play an important role in our main result which provides a first step toward a complete solution of Gizatullin's problem for quartic surfaces with Picard number two.

\begin{thm*}
    Let $S$ be a K3 quartic surface with $\rho(S)=2$ such that $|disc(S)|> 233$.
Then no non-trivial element of $Aut(S)$ is derived from $Bir(\mathbb{P}^{3} )\backslash Aut(\mathbb{P}^{3})$ under any embedding $S \hookrightarrow \mathbb{P}^{3}$. Moreover, no non-trivial automorphism in $ G$ is induced by an element of $Bir(\PP^3)$. 
\end{thm*}
Gizatullin's problem for quartic surfaces with $\rho(S)=2$ and $|disc(S)|\leq233$ will be addressed in the forthcoming paper \cite{APZ}.\\
    
In \cite{Oguiso2}, Oguiso constructed a smooth quartic K3 surface $S$ in $\mathbb{P}^3$ with $\rho(S)=3$ such that every element of $Aut(S)$ is derived from a Cremona transformation of the ambient space. In this case $Aut(S)=\mathbb{Z}_2\ast\mathbb{Z}_2\ast\mathbb{Z}_2$, in particular, $Aut(S)$ is generated by elements of finite order. 
Hence, a natural question posed by Oguiso in \cite{Oguiso2} is

\begin{Q} \label{Oguisoquestion}
    Is every automorphism of finite order of any smooth quartic surface $S\subset \mathbb{P}^3$ induced by a Cremona transformation of $\mathbb{P}^3$?
\end{Q}

As a consequence of the main result we obtain a negative answer to Oguiso's question. Indeed, by using the classical theory of binary quadratic forms, Galluzzi, Lombardo and Peters proved in \cite{Galluzzi} that for a K3 surface $S$ with Picard number two the subgroup $G\subset Aut(S)$ falls into four possibilities: trivial, $\mathbb{Z}_2$, $\mathbb{Z}$ or $\mathbb{D}_{\infty}=\mathbb{Z}_{2} \ast \mathbb{Z}_{2}$. Then, it is sufficient to take quartic surfaces where $G=\mathbb{D}_{\infty}$ or $G=\ZZ_2$. Moreover, if the K3 quartic is general (see \ref{autgeral}), the group $G$ is the whole $Aut(S)$ and so, if it satisfies the condition of our main result, Gizatullin's problem is completely solved. More concretely, we construct examples of quartic K3 surfaces with Picard number two
 such that $Aut(S) = \mathbb{D}_{\infty}=\mathbb{Z}_2 * \mathbb{Z}_2 $ or  $Aut(S)= \mathbb{Z}_2$ and no nontrivial element of $Aut(S)$ is induced by Cremona transformations of $\mathbb{P}^3$. In particular, these surfaces contain automorphisms of order two which are not derived from $Bir(\PP^3)$.

The structure of the paper is the following. In section 2, we recall basic definitions and results about K3 surfaces and lattices that will be used in the paper. In section 3, we prove some results for  quartics in $\mathbb{P}^{3}$ with Picard number two.
In section 4, we construct some first results for Gizatullin's problem when the quartic surface has Picard number two. In the last section, we construct our examples.

Throughout this note, we work over the field $\mathbb{C}$ of complex numbers.\\

\textbf{Acknowledgements.} This paper grew from our Ph.D. research work. We would like to thank our Ph.D. advisor Carolina Araujo for presenting us Gizatullin's problem and for the several fruitful discussions that culminated in this article. We also thank Alice Garbagnati and the anonymous referees for their valuable suggestions and careful reading. We extend our gratitude to Paola Comparin, Michela Artebani, Alessandra Sarti, Giacomo Nanni, Pablo Quezada, Roberto Villaflor for insightful discussions related to this problem. Furthermore, We would like to thank CAPES (Coordernação de Aperfeiçoamento de Pessoal de Nível Superior) for the financial support that made this paper possible. 

\section{Preliminaries}
In this short section, we set the notation and state some basic
results that will be used in the paper.
\subsection{K3 surfaces}\label{K3surfaces}
We consider a $K3$ \emph{surface} to be a smooth projective algebraic surface $S$ such that $h^{1}(S,\mathcal{O}_S)=0$ and the canonical divisor $K_S$ is trivial, i.e., there is a nowhere vanishing global 2-form $\omega_S$. This 2-form is a generator of $H^0(S,\Omega_S^2)=\mathbb{C}\omega_S$. 
Examples of K3 surfaces include smooth complete intersections with trivial canonical divisor, such as smooth
quartic surfaces in $\mathbb{P}^3$.

For a K3 surface $S$, the group $NS(S)$ of divisors modulo numerical equivalence, the \emph{Néron-Severi group}, is isomorphic to the group $Pic(S)$ of isomorphisms classes of line bundles over $S$, the \emph{Picard group}.
As a matter of fact, the exact sequence
$$0 \longrightarrow \mathbb{Z} \longrightarrow \mathcal{O}_{S} \longrightarrow \mathcal{O}^{*}_{S} \longrightarrow 1 $$
induces a map $c_1 \colon H^{1}(S, \mathcal{O}^{*}_{S}) \longrightarrow H^{2}(S, \mathbb{Z})$. We have that $Pic(S) \cong H^{1}(S, \mathcal{O}^{*}_{S}) $ and the Néron-Severi group can be defined as the image of this map. Since $H^{1}(S, \mathcal{O}_{S})=0$, the map $c_1$ is injective and we have the desired isomorphism.
Therefore, from now on, we identify $Pic(S)$ with a subgroup of the cohomology group $H^2(S,\mathbb{Z})$, where the intersection form of $Pic(S)$ coincides with the restriction of the cup product. 
The Picard group $Pic(S)$ is a  finitely generated abelian group and its rank $\rho(S)$ is called the \emph{Picard number} of $S$. 

The \emph{positive cone} $P(S)$ is the connected component of $\{x\in Pic(S)\otimes_{\mathbb{Z}} \mathbb{R}|x^2>0\}$ which contains ample classes. The \emph{ample cone} $Amp(S)\subset P(S)$ is the convex cone spanned, over $\mathbb{R}$, by ample line bundles. The next proposition establishes a special relationship for these two cones when $S$ is a projective K3 surface. A reference for this result is \cite[Corollary 1.7, Chapter 8]{Huybrechts}. A ($-2$)-curve on $S$ is an irreducible curve $C$ with $C^2=-2$. In particular, by adjunction, $C\cong\mathbb{P}^1$ . More generally, for $n \in \mathbb{N}$, an $n$-divisor is a divisor D such that $D^{2}= n$.
\\

\begin{prop}\label{amplecriterio}For a projective K3 surface $S$,$$Amp(S)= \{ x \in P(S)| x \cdot C > 0 \text{ for all ($-2$)-curves $C$} \}. $$\end{prop}

Given a nef line bundle
$H \in Pic(S)$, the next theorem determines necessary and sufficient conditions for $H$ to be very ample. This result is a combination of results proven in \cite{SaintDonat}, but the following formulation can be found in \cite[Theorem 5]{Mori}.

\begin{thm}\label{Mori}
     Let $S$ be a K3 surface and $H$ a nef line bundle with $H^2\geq 4$. Then, $H$ is very ample if and only if it satisfies the following 3 conditions:
     \begin{itemize}
      
      \item There is no irreducible curve $E$ such that $E^{2}=0$ and
      $H \cdot E\in\{1,2\}$.
      \item There is no irreducible curve $E$ such that $E^{2}=2$, and $H \sim 2E$.
      \item There is no irreducible curve $E$ such that $E^{2}=-2$ and
      $E \cdot H=0$.
         
     \end{itemize}
\end{thm}

\subsection{Lattices}
A \emph{lattice} $L$ is a free $\mathbb{Z}$-module with finite rank together with a symmetric bilinear form $(\ ,\ ):L\times L\longrightarrow\mathbb{Z}$. A lattice is called \emph{even} if $(x,x)\in 2\mathbb{Z}$ for all $x\in L$. The determinant of an intersection matrix is called the \emph{discriminant} of $L$ denoted by $discr(L)$. If $discr(L)\neq0$, $L$ is called \emph{non-degenerate} and if $discr(L)=\pm 1$ it is called \emph{unimodular}. The bilinear form $(\ ,\ )$ can be extended to a symmetric bilinear form on the real vector space $L\otimes\mathbb{R}$. Assuming the lattice $L$ is non-degenerate, the bilinear form on $L \otimes \mathbb{R}$ is equivalent
to that represented by a diagonal matrix whose nontrivial entries are either 1 or $- 1$. The \emph{signature} of $L$ is the pair $(n_+,n_-)$, where $n_{\pm}$ is the number of $\pm 1$ on the diagonal. An example of an even lattice is the Picard group $Pic(S)$ of a K3 surface $S$ together with its intersection form, which is called the \textit{Picard Lattice of S}. We denote by $disc(S)$ the discriminant of $Pic(S)$.
As a matter of fact, Morrison proved in \cite[Corollary 2.9]{Morrison} the following result:

\begin{thm} \label{existenciak3}
Every even lattice $L$ of signature (1, $\rho -1$) with $\rho \leq 10$ occurs as the Picard Lattice of some K3 surface.
\end{thm}
More precisely, there exists a $(20-\rho)$-dimensional
family of K3 surfaces with Picard lattice isomorphic to $L$.\\

A lattice $L$ is called \emph{$p$-elemetary}, for a prime $p$, if $A(L)\cong (\ZZ\slash p\ZZ)^a$, for some positive integer $a$. 
The \emph{orthogonal group} $O(L)$ is the group of all isomorphisms of $L$ preserving the bilinear form, the so-called \emph{isometries}. 
The \emph{dual lattice} of $L$ is $$L^*=\{x\in L\otimes\mathbb{Q}|(x,y)\in\mathbb{Z}\ \text{for any }y\in L\}.$$We have that $L\subset L^*$ and the quotient $A(L):=L^{*}/L$ is called the \emph{discriminant group}. Any isometry $\varphi$ on $L$ induces an automorphism $\bar{\varphi}$ on $A(L)$. \\
Suppose that $L$ is an even {unimodular} lattice. A sublattice $L_1$ is called \emph{primitive} if $L/L_1$ is free. In this case, taking $L_2$ as the orthogonal complement of $L_1$ in $L$, we have a natural isomorphism $A(L_1)\cong A(L_2)$.
\begin{prop}
    Let $L,L_1$ and $L_2$ be as above. Suppose that $\varphi\in O(L)$ preserves $L_1$ and $L_2$, so the restriction $\varphi_i=\varphi|_{L_i}$ to $L_i$ is an element of $O(L_i)$ for $i=1,2$. Then, $\bar{\varphi}_1=\bar{\varphi}_2$ under the identification $A(L_1)\cong A(L_2)$.   
\end{prop}
The following theorem, which can be found in \cite[Corollary 1.5.2]{Nikulin}, gives the converse.

\begin{thm}[Gluing isometries]\label{gluing} Let $L,L_1$ and $L_2$ be as above and let $\varphi_1,\varphi_2$ be isometries of $L_1$ and $L_2$, respectively. If $\bar{\varphi}_1=\bar{\varphi}_2$ under the identification $A(L_1)\cong A(L_2)$, then there exists an isometry $\varphi$ on $L$ whose restrictions to $L_1$ and $L_2$ are $\varphi_1$ and $\varphi_2$, respectively.    
\end{thm}

\subsection{Automorphisms of K3 surfaces}\label{AutK3}

For a projective K3 surface $S$, $H^2(S,\mathbb{Z})$ with the cup product and $Pic(S)$ with the intersection form are even and non-degenerate lattices with signatures $(3,19)$ and $(1,\rho(S)-1)$, respectively. {Under the identification $Pic(S)\subset H^2(S,\ZZ)$ in $\S$\ref{K3surfaces}}, we view $Pic(S)$ as the primitive sublattice $\{x\in H^2(S,\ZZ)|(x,\omega_S)=0\}$ of $H^2(S,\mathbb{Z})$. The orthogonal complement of $Pic(S)$ in $H^2(S,\mathbb{Z})$ is called the \emph{transcendental lattice} of $S$ and is denoted by $T(S)$. As we have seen in the previous subsection, there is a natural isomorphism $A(Pic(S))\cong A(T(S))$. Moreover, $H^2(S,\mathbb{Z})$ and $T(S)$ have a weight two Hodge structure $H^2(S,\mathbb{C})=H^{2,0}(S)\oplus H^{1,1}(S)\oplus H^{0,2}(S)$ and $T_\mathbb{C}(S):=T(S)\otimes \mathbb{C}=T^{2,0}(S)\oplus T^{1,1}(S)\oplus T^{0,2}(S)$, respectively, where $H^{0}(S,\Omega_S^2)=H^{2,0}(S)=T^{2,0}(S)\subset T_\mathbb{C}(S)$.\\ 

An automorphism $g\in Aut(S)$ induces an isometry $g^{*}$ of $H^{2}(S,\mathbb{Z})$ preserving $Pic(S)$, $T(S)$ and induces an isomorphism $g^{*}_\mathbb{C}$ of $H^{2}(S,\mathbb{C})$ preserving its Hodge structure. The next result tells us that studying automorphisms of $S$ can be translated into studying isometries of the lattice $H^2(S,\mathbb{Z})$. A reference for this theorem is \cite[Chapter 7, Theorem 5.3]{Huybrechts}

\begin{thm}[Global Torelli Theorem]\label{Torelli} Let $\varphi:H^{2}(S,\mathbb{Z})\longrightarrow H^2(S,\mathbb{Z})$ be an isometry that preserves the ample cone $Amp(S)$ and such that $\varphi_\mathbb{C}(H^{2,0}(S))=H^{2,0}(S)$. Here, $\varphi_{\mathbb{C}}$ denotes the natural extension of $\mathbb{\varphi}$ to $H^{2}(S, \mathbb{C})$. Then there exists a unique automorphism $g$ of $S$ such that $g^*=\varphi$.
    
\end{thm}

An automorphism $g\in Aut(S)$ is called \emph{symplectic} if it acts trivially on $H^{2,0}(S)$ i.e. $g^*_\mathbb{C}(\omega_S)=\omega_S$, and it is called \emph{anti-symplectic} if $g^*_\mathbb{C}(\omega_S)=-\omega_S$.
Any isometry of $T(S)$ preserving its Hodge structure is determined by its action on $H^{2,0}(S)$ (see \cite[Chapter 3, Lemma 3.3]{Huybrechts}). Moreover, an automorphism $g$ of $S$ is symplectic (respectively, anti-symplectic) if and only if $\bar{g}^*=Id$ (respectively, $\bar{g}^*=-Id$)  on $A(T(S))$.
Moreover, if $\bar{g^{*}} $ acts either as $Id$ or as $-Id$ on $A(T(S))$, the same is true on $A(Pic (S))$. When an automorphism has order two  the following holds (see \cite[\S 15.1, pg.\ 310]{Huybrechts} and \cite[Lemma 1]{Bini}).
\begin{prop}\label{simpleticinv}
    If $g$ is a symplectic order two automorphism of a K3 surface $S$, then $\rho(S)>8$.
\end{prop}

\section{Quartic surfaces with Picard rank two}
In this section we will explore properties about smooth quartic surfaces with Picard number two and their automorphisms. We start by proving that a basis of the Picard lattice containing the hyperplane class can always be chosen.

\begin{lem}\label{PicK3}
Let $S$ be a K3 quartic surface with $\rho(S)=2$ and  $i_{H}\colon S \hookrightarrow \mathbb{P}^{3}$ the embedding induced by a very ample divisor $H$, $H^{2}=4$.
Then, we can write
$Pic(S)=\mathbb{Z}H\oplus\mathbb{Z}W$ for some divisor class $W$ and thus the intersection product is given by the following matrix
\begin{equation} \label{intermatrix}
Q =\begin{pmatrix}

 4& b \\ 
b & 2c
\end{pmatrix}.\end{equation} 
\end{lem}
\begin{proof}
    Let $\{U,V\}$ be a basis of $Pic(S)$. We notice that $H$ is a primitive element of the lattice, otherwise if $H=kZ$, where $k \in \mathbb{Z}$ and $Z\in Pic(S)$, $k^{2}Z^{2}=4$, which is only possible for $k=1$ since $Z^{2}$ is an even number.

   Therefore, we can write $H= \alpha U + \beta V$, where $g.c.d(\alpha, \beta)=1$ and therefore there are integers $\gamma$ and $\delta$ such that $\delta \alpha+ \gamma \beta =1$. Let $W:= -\gamma U+ \delta V $.
The matrix
\begin{equation*}
 A=\begin{pmatrix}
 \alpha & \beta \\ 
- \gamma & \delta 
\end{pmatrix}
\end{equation*} 
is invertible over $\mathbb{Z}$, which implies that $\{ H, W \}$ is a basis of $Pic(S)$.
After this change of basis, the intersection matrix is the desired one. 
\end{proof}
Next, we establish a necessary and sufficient condition for these surfaces to have $2k$-divisors. We give a more general statement for rank two lattices.
\begin{lem} \label{latticecondition}Consider a lattice  $L=\mathbb{Z}h_1 \oplus\mathbb{Z}h_2$ with bilinear form given by the matrix $Q$ in \eqref{intermatrix}.
Let $r := - discr(L)$ and $k\neq 0$ an integer number. Then, the following assertions hold:
\begin{enumerate}
\item {For any $v\in L$ we have that $4v^2=(v\cdot h_1)^2-rn^2$, for some integer $n$.}
    \item There are vectors $v \in L$ such that $v^{2}=0$ if and only if r is a square;
    \item There are vectors $v \in  L$ such that $v^2 = 2k$ if and only if the Generalized Pell equation $x^2 - ry^2 = 8k$ has integer solutions. 
\end{enumerate}
\end{lem}
\begin{proof}
{Note that assertions $(2)$ and $(3)$ follow from $(1)$. Now, we prove $(1)$.\\}
Let $v= mh_1+nh_2$ for some $n,m \in \mathbb{Z}$.
Then $d:= v \cdot h_1= 4m+bn\Longrightarrow m= \dfrac{d-bn}{4}.$
So, we can write
$$v^{2}= \dfrac{d^{2}-2bnd+(bn)^{2}}{4}+\dfrac{bnd-(bn)^{2}}{2}+2cn^{2}$$
and so
$$4v^{2}=d^{2}-n^{2}(b^{2}-8c)\Longrightarrow 4v^2=d^{2}-rn^{2},$$
which proves the result.
\end{proof}
 By Theorem \ref{existenciak3}, given an even lattice of
signature $(1,1)$, it occurs as the Picard Lattice of some K3 surface. In the following result, we add conditions on the lattice so that this K3 surface is in fact a quartic in $\mathbb{P}^{3}$.
\begin{prop}\label{k3quartica}
Let $L$ be a rank two lattice of signature $(1,1)$, with $|disc(L)|>8$ and representing $4$ (i.e., there is an element $D\in L$ such that $D^2=4$). Then, for any K3 surface $S$ such that $Pic(S)\cong L$, there is a very ample divisor $H \in Pic(S)$, such that $H^{2}=4$ and the induced embedding $i_{H}\colon S \hookrightarrow \mathbb{P}^{3} $ is an isomorphism onto a smooth quartic surface.
\end{prop} 

\begin{proof}
Let $S$ be a $K3$ surface such that $Pic(S)\cong L$ and $H\in Pic(S)$ with $H^2=4$.
We want to apply Theorem \ref{Mori} to guarantee that $H$ is a very ample line bundle and therefore induces and embedding  of $S$ as a quartic in $\mathbb{P}^{3}$.
For that, first we have to prove that
$H$ is a nef divisor.
We consider $\Delta= \{b \in Pic(S) \text{ | } b^{2}=-2 \text{ and $b$ is effective} \}$, given $b \in \Delta$, we can define a Picard-Lefchetz reflection $s_b$ as
\begin{align*}
    s_b \colon Pic(S) &\to Pic(S)  \\
             D &\mapsto D+(b\cdot D)b
\end{align*}

By \cite[chapter VIII, proposition 3.9]{BCV}, the closure $\overline{Amp(S)} \cap P(S)$ is a strict fundamental domain for the action of the group Picard-Lefchetz reflections on $P(S)$.
 This means that, given $s_b(H)$,
 there is a nef divisor $D'$ and $b' \in \Delta$, such that $s_b(H)= s_b'(D')$, since the Picard-Lefchetz reflections are isometries,
 we have $H^{2}=D'^{2}=4$ and, after replacing $H$ with $D'$ if necessary, we can
 suppose that $H$ is nef.
 
Since $H$ is a primitive element of the lattice, $Pic(S)=\langle H,W\rangle$, for some element $W$, and the intersection matrix is given by $Q$ in (\ref{intermatrix}). Note that the second condition of Theorem \ref{Mori} is immediately satisfied. Suppose there is an irreducible curve $E$ with $(E^{2},H \cdot E)\in\{(0,1),(0,2),(-2,0)\}$. Consider the sublattice $P:=\langle H, E\rangle$ of $Pic(S)$.
This is a rank two lattice with discriminant $4E^2-(H\cdot E )^2=-1,-4,-8$, respectively. By Lemma \ref{latticecondition},  $4E^2=(H\cdot E)^2-rn^2$, for some $n\in\ZZ$. This implies that $r$ divides $1,4$ and $8$ respectively. We arrive to a contradiction since $r>8$.
Therefore, we conclude that all conditions of Theorem \ref{Mori} are satisfied and $H$ is very ample. Finally, the image of $i_H$ is a quartic of $\PP^3$ since $H^2=4$. 
\end{proof}

\subsection{Automorphisms of K3 surfaces with Picard number two}

Let $S$ be a smooth quartic surface. We define $G\subset Aut(S)$\label{group G} as the subgroup of all automorphisms $g$ such that $g^*\omega_S=\pm\omega_S$. The important property of this group is that we can determine it via isometries of $Pic(S)$. This follows from Theorem \ref{gluing} and Theorem \ref{Torelli}. Indeed, every element of $G$ corresponds to an isometry of $Pic(S)$ preserving $Amp(S)$ and whose action on $A(Pic(S))$ is $\pm Id$. By \cite[Theorem 10.1.2]{Nikulin3}, $G$ is a finite index subgroup of $Aut(S)$. If $\rho(S)=2$, the group $G$ is described by the theorem below, which was proved by Galluzzi, Lombardo, and Peters in \cite[Proposition 3, Proposition 4]{Galluzzi}.
 
 \begin{thm}\label{GLP}
Let $S$ be a smooth quartic surface with Picard number two and let $G\subset Aut(S)$ be the subgroup defined above. If there is a nontrivial divisor $D$ on $S$ such that $D^{2}=0$ or $D^{2}=-2$, then $G$ is either trivial or $\mathbb{Z}_2$. If $S$ contains no such divisors, then $G$ is either  $\mathbb{Z}$ or
$\mathbb{D}_{\infty}=\ZZ_2\ast\ZZ_2$.
\end{thm}

By this previous result, non-trivial automorphisms of $G$ are either of order two or of infinite order. Next, we give a necessary and sufficient condition for $S$ to admit involutions. For an automorphism $g$ we define the invariant lattice $L_g:=\{x\in H^2(S,\ZZ)|g^*x=x\}$.

\begin{prop}\label{nikinvolutions}Let $S$ be a smooth quartic surface with $\rho(S)=2$. Then, $S$ admits an involution $g$ if and only if there exists an ample divisor $D$ with $D^2=2$. In this case, $L_g=\langle D\rangle$.   
\end{prop}
\begin{proof}
    Assume that $S$ admits an involution $g$. Since it has order two, $g^*\omega_S=\pm\omega_S$. By Proposition \ref{simpleticinv},  $g^*\omega_S=-\omega_S$. Notice that $L_g\subset Pic(S)$ and it has rank one. Indeed, for any element $x\in L_g$ the equality $\langle x, \omega_S\rangle=\langle g^*x, g^*\omega_S\rangle=-\langle x, \omega_S\rangle$ implies that $\langle x,\omega_S \rangle=0$
    and so $x\in Pic(S)$. Using \cite[Theorem 1.2]{Lee}, we can compute that $1$ and $-1$ are the eigenvalues of $g^*$ on $Pic(S)$ and so $L_g$ has rank one and it is generated by a divisor $D$. For any ample class $A$, $A+g^*A$ is ample and it is preserved by $g^*$. Then, $D$ is ample since $A+g^*A$ is a multiple of it. Then, $D^{2}=2m>0$, and $\frac{D}{2m} \in A(L_g)$, which is a contradiction for $m\neq 1$, since by \cite[Theorem 2.1]{AST} , $L_g$ is $2$-elementary, so $D^2=2$.

    Conversely, assume that there is an ample divisor $D\in Pic(S)$ such that $D^2=2$. Since $S$ is a smooth K3 surface, results of Saint-Donat \cite[Prop 2.6, Thm 3.1, 5.1]{SaintDonat} establish that  the rational map induced by the complete linear system $|D|$ identifies $S$ as a double cover of $\mathbb{P}^2$ branched along a sextic curve and it admits in a natural way an involution. 
\end{proof}

\begin{defi}
 A K3 surface $S$ admitting only symplectic or antisymplectic automorphisms is called \emph{Aut-general}.
\end{defi}

\begin{rmk}\label{autgeral} By Theorem \ref{existenciak3}, for any even lattice $L$ with signature $(1,1)$ there is a K3 surface $S$ with Picard lattice $L$, and we can always take such surface to be Aut-general. In fact, a general member in the $18$-dimensional family of K3's with Picard lattice $L$ is Aut-general (see for instance \cite[Theorem 10.1.2 c)]{Nikulin3}). In this case, the group $G$ is the whole group $Aut(S)$.
\end{rmk}

\section{Gizatullin's problem in Picard rank two}
In this section, we explore Gizatullin's problem for quartic K3 surfaces with Picard number two.
Our goal is to find a sufficient condition on the Picard lattice so that no element of infinite order of $Aut(S)$ is derived from $Bir(\mathbb{P}^{3})$ under any embedding $ S\hookrightarrow \mathbb{P}^3$. With this in mind, we recall a result by Oguiso \cite[Theorem 3.2]{Oguiso1}. The second part follows from a result of Takahashi \cite{Takagashi}.

\begin{thm}\label{oguiso}
     Let $S \subset \mathbb{P}^3$ be a smooth quartic K3 surface. Then:
     \begin{itemize}
         \item Any automorphism $g$ of infinite order is not the restriction of an automorphism of the ambient space $\mathbb{P}^3$.
         
         \item Assume that $S$ contains no curves of degree less than 16 except for complete intersections of the surface $S$ with hypersurfaces of $\mathbb{P}^3$. Then no automorphism of $S$ is the restriction of an element of $Bir(\mathbb{P}^{3}) \backslash Aut(\mathbb{P}^{3})$.
     \end{itemize}
\end{thm}
As a consequence of the above theorem, we obtain a criterion for Gizatullin's problem using the discriminant of the Picard lattice.
 
\begin{thm}\label{PaivaQuedoidea} 

Let $S$ be a K3 quartic surface with $\rho(S)=2$ such that $-disc(S)=r> 233$.
Then no element of $Aut(S)$ is derived from $Bir(\mathbb{P}^{3} )\backslash Aut(\mathbb{P}^{3})$ under any embedding $ S \hookrightarrow \mathbb{P}^{3}$.
Moreover, no element of infinite order of $Aut(S)$ comes from  $Bir(\mathbb{P}^{3} )$.
\end{thm}
\begin{proof} Let $\phi: S \hookrightarrow \mathbb{P}^3$ be an embedding of $S$ in $\mathbb{P}^3$ and  ${H}$ the hyperplane class on $S$ under this embedding. 
    We will prove that every curve $C \subset S$ that has degree $d <16$ with respect to the embedding $\phi$ is of the form $C = S\cap T$, where $T$ is a hypersurface of $\mathbb{P}^{3}$.

Let $C$ be a curve on $S$ with degree $d:=H\cdot C<16$. By Lemma \ref{PicK3}, and Lemma \ref{latticecondition}, $C=mH+nW$ for some divisor $W$ and integers $m,n$; and $4C^2=d^2-rn^2$. From Riemann-Roch we know that $-2\leq 2p_a(C)-2=C^2$, where $p_a(C)$ is the arithmetic genus of $C$. Then
$$233n^{2}<d^{2}+8. $$

Since $d\leq 15$, we have $d^{2}+8 \leq 233$. As $n $ is an integer number, the above inequality is satisfied only if $n=0$. Therefore, $C=mH$ in $Pic(S)$.

The following exact sequence

$$0 \longrightarrow \mathcal{O}_{\mathbb{P}^{3}}(-4+m) \longrightarrow \mathcal{O}_{\mathbb{P}^{3}}(m) \longrightarrow \phi_{*}\mathcal{O}_{S}(m)\longrightarrow 0$$
induces the following one
$$H^{0}(\mathbb{P}^{3}, \mathcal{O}_{\mathbb{P}^{3}}(m)) \longrightarrow H^{0}(\mathbb{P}^{3},\phi_{*}\mathcal{O}_{S}(m))\longrightarrow H^{1}(\mathbb{P}^{3}, \mathcal{O}_{\mathbb{P}^{3}}(-4+m)). $$
Since $H^{1}(\mathbb{P}^{3}, \mathcal{O}_{\mathbb{P}^{3}}(l))=0\  $ for all $l \in \mathbb{Z}$, the map $H^{0}(\mathbb{P}^{3}, \mathcal{O}_{\mathbb{P}^{3}}(m)) \longrightarrow H^{0} (\mathbb{P}^{3}, \phi_{*}\mathcal{O}_{S}(m))$ is surjective, which implies that $S$ contains no curve of degree less than 16 which is not the complete intersection of the surface $S$ and a hypersurface $T$ of $\mathbb{P}^3$. Using Theorem \ref{oguiso}, we have the desired result.

\end{proof}
\begin{cor}\label{answeroguiso}
  Let $S$ be a K3 quartic surface as in the above theorem and let $G\subset Aut(S)$ be the subgroup define in $\S$\ref{group G}. Then, no non-trivial element of $G$ comes from $Bir(\mathbb{P}^{3})$ under any embedding $S \hookrightarrow \mathbb{P}^{3}$. Moreover, if
  $S$ is Aut-general, the same is true for the whole group $Aut(S)$.
\end{cor}

\begin{proof}
 First, we recall that, by Theorem \ref{GLP}, the possibilities for $G$ are: trivial, $\ZZ_2$, $\ZZ$ and $\mathbb{D}_{\infty}$. The proof is obvious if $G$ is trivial. If $G=\ZZ$, the result follows immediately from Theorem \ref{PaivaQuedoidea}. Thus, let us assume that $G=\ZZ_2$ or $G=\mathbb{D}_{\infty}$.
 
 Let $\phi: S \hookrightarrow \mathbb{P}^3$ be an embedding of $S$ in $\mathbb{P}^3$ and $H$ the very ample divisor inducing $\phi$. Suppose by contradiction there is a non-trivial automorphism $g\in G$ that is induced by an element of $Bir(\mathbb{P}^{3})$.
By Theorem \ref{PaivaQuedoidea}, it has finite order and it must be induced by an element of  $Aut(\mathbb{P}^{3})$.  Then,  $g$ is an involution preserving the hyperplane class, i.e., $g^*H=H$. By Proposition \ref{nikinvolutions}, 
$H$ is a multiple of a divisor $D$ with $D^2=2$, which is contradiction since $H$ is a primitive element of the Picard lattice. 

Finally, if $S$ is Aut-general, the result follows from the fact that $Aut(S)=G$. 
\end{proof}

\section{Examples}

In this section, we describe examples of smooth quartic surfaces for which no nontrivial automorphism of the surfaces is induced by any Cremona transformation under any embedding. In particular, these surfaces have finite order automorphisms for which the answer to Question \ref{Oguisoquestion}, posed by Oguiso, is negative. More precisely, on the one hand, we construct a family $S_n \subset \mathbb{P}^{3}$ of smooth quartic surfaces with $\rho(S_n)=2$ such that $Aut(S_n)\cong\mathbb{D}_\infty$ and, on the other hand, we construct a rank two smooth quartic surface $S\subset\PP^3$ with
$Aut(S)\cong\ZZ_2$.
\begin{ex}
    For $n\geq 2$,  consider the rank two lattice $L_n$ with intersection matrix given by 

\begin{equation*}
 Q_n=\begin{pmatrix}
4 & 8n \\
8n & 2 
\end{pmatrix}.
\end{equation*}
Note that $r:=-disc(L_n)=8(8n^2-1)>233$. Since $L_n$ satisfies the hypothesis of Proposition \ref{k3quartica}, there exists a smooth quartic surface $S_n$ such that $Pic(S_n)\cong L_n$. Moreover, we can choose $S_n$ to be Aut-general by Remark \ref{autgeral}. By Corollary \ref{answeroguiso}, no nontrivial element of $Aut(S_n)$ is induced by an element of $Bir(\mathbb{P}^3)$ under any embedding $S_n\hookrightarrow \mathbb{P}^3$. It remains to see that $Aut(S_n)=\mathbb{D}_{\infty}$ and for this we use Lemma \ref{latticecondition}.  We have that $r$ is not a square number,
and the Generalized Pell equation $ x^{2}- ry^{2}=-8$ does not have solutions. Indeed, by rewriting the equation as $ x^{2}-8(8n^{2}-1)y^{2}=-8$ we note that $x^{2}$ is a multiple of 8. Since it is a square number, it is also a multiple of 16 and so $x^{2}=16x'^{2}$, for some integer $x'$. Thus, the equation $x^{2}-8(8n^{2}-1)y^{2}=-8$ has integer solutions if and only if 
    \begin{equation}
    \label{eq}
            2x'^{2}-(8n^{2}-1)y^{2}=-1
    \end{equation}
     has integer solutions.
     Let us see that the last equation does not have solutions modulo 8. We recall that there are 3 possibilities for a square number modulo 8:
    \begin{itemize}
        \item If $a\equiv 2 \text{ or } 6 \mod{8} \Longrightarrow a^{2} \equiv  4 \mod{8}$.
        \item If $a\equiv 0 \text{ or } 4 \mod{8} \Longrightarrow a^{2} \equiv  0 \mod{8}$.
        \item If $a$ is odd  $\Longrightarrow a^{2} \equiv  1 \mod{8}$.
    \end{itemize}
If $x'$ is an even number, the equation (\ref{eq}) is reduced to $y^{2}=7 \mod{8}$, and if $x'$ is an odd number, the equation (\ref{eq}) is reduced to $y^{2}=5 \mod{8}$. In both cases the equation does not have solutions. 

Therefore, the Picard lattice does not have elements with self-intersection $0$ or $-2$, and clearly there is a divisor $D$ with $D^2=2$. Then, $D$ or $-D$ is ample by Proposition \ref{amplecriterio}. Therefore, the assertion follows from Theorem \ref{GLP} and Proposition \ref{nikinvolutions}.  
\end{ex}

\begin{ex}
   In a similar way, consider the rank two lattice $L$ with intersection matrix given by

\begin{equation*}
 Q=\begin{pmatrix}
4 & 17 \\
17 & 2 
\end{pmatrix}
\end{equation*}
and denote by $r:= - discr(L)=281>233$.
We can prove that there is a quartic K3 surface $S\subset\mathbb{P}^3$ such that $L$ is its Picard lattice, $Aut(S)=\mathbb{Z}_{2}$
and no nontrivial element of $Aut(S)$ is induced by an element of $Bir(\mathbb{P}^3)$ under any embedding $S\hookrightarrow \mathbb{P}^3$.  
The reasoning is similar to the previous case, the only difference is that, in this case, the Generalized Pell equation $x^{2}-281y^{2}=-8$ has solutions, for example (251999,15033).\\
\end{ex}

\end{document}